\documentclass[12pt,reqno]{amsart}

\usepackage{amsfonts}
\usepackage{amsmath}
\usepackage{amssymb}
\usepackage{amsthm}
\usepackage{bm}
\usepackage{enumerate} 
\usepackage{hyperref}
\usepackage{cite}
\usepackage[margin=1.2in]{geometry}
\usepackage{mathrsfs} 
\usepackage{graphicx} 
\usepackage{tikz-cd}

\usepackage{xcolor}

\theoremstyle{plain}
\newtheorem{theorem}{Theorem}[section]
\newtheorem{lemma}[theorem]{Lemma}
\newtheorem{proposition}[theorem]{Proposition}
\newtheorem{corollary}[theorem]{Corollary}

\theoremstyle{definition}

\newtheorem{remark}[theorem]{Remark}

\newtheorem{example}[theorem]{Example}

\numberwithin{equation}{section}

\newcommand\N{\mathbb{N}}

\newcommand\R{\mathbb{R}}

\newcommand\C{\mathbb{C}}

\renewcommand\S{\mathcal{S}}

\renewcommand\H{\mathcal{H}}

\newcommand{\maclA}{\mathcal{A}}

\newcommand\B{\mathfrak{B}}

\newcommand\F{\mathcal{F}}

\newcommand\ev[2]{\langle#1,#2\rangle}

\DeclareMathOperator{\spn}{span}

\begin{document}

\title[Spaces of functions with nearly optimal time-frequency decay]{Hermite expansions for spaces of functions with nearly optimal time-frequency decay}

\author[L. Neyt]{Lenny Neyt}
\thanks{L. Neyt gratefully acknowledges support from the Alexander von Humboldt Foundation and the Research Foundation--Flanders through the postdoctoral grant 12ZG921N}
\address{Universit\"{a}t Trier\\ FB IV Mathematik\\ D-54286 Trier\\ Germany}
\email{lenny.neyt@UGent.be}

\author[J. Toft]{Joachim Toft}
\thanks{J. Toft was supported by Vetenskapsr{\aa}det, project number 2019-04890.}
\address{Department of Mathematics \\ Linn\ae us University \\  Växjö \\ Sweden}
\email{joachim.toft@lnu.se}

\author[J. Vindas]{Jasson Vindas}
\thanks{The work of J. Vindas was supported by the Research Foundation--Flanders, through the FWO-grant number G067621N}
\address{Department of Mathematics: Analysis, Logic and Discrete Mathematics \\ Ghent University \\ Krijgslaan 281 \\ 9000 Gent \\ Belgium}
\email{jasson.vindas@UGent.be}

\subjclass[2020]{ 46E10; 42A38; 42C10; 30H20}
\keywords{Functions with nearly optimal time-frequency decay; Bargmann transform; Phragm\'{e}n-Lindel\"{o}f principle on sectors; Hermite series expansions.}

\begin{abstract}
We establish Hermite expansion characterizations for several subspaces
of the Fr\'{e}chet space of functions on the real line satisfying
\begin{equation*}
|f(x)| \lesssim e^{-(\frac{1}{2} - \lambda ) x^{2}} , \qquad |
\widehat{f}(\xi )| \lesssim e^{-(\frac{1}{2} - \lambda ) \xi ^{2}} ,
\qquad \forall \lambda > 0 .
\end{equation*}
In particular, we extend and improve Fourier characterizations
of the so-called proper Pilipovi\'{c} spaces obtained in \cite{T-G-FourierCharPilipovicSp}. The main ingredients in our proofs are the Bargmann
transform and some achieved optimal forms of the Phragm\'{e}n-Lindel\"{o}f principle.
\end{abstract}

\maketitle

\section{Introduction}
A convenient way to characterize the smallest Fourier invariant classical Gelfand-Shilov space
$\mathcal{S}_{1/2} (\mathbb{R})$ \cite{Gelfand-ShilovII} is given by Gaussian estimates of the
involved functions and their Fourier transforms. In fact, in \cite{Gelfand-ShilovII,CCK96}
it is proven that  a function or distribution $f$ on $\mathbb R$ belongs
to $\mathcal{S}_{1/2} (\mathbb{R})$ 
if and only if
\begin{equation}
\label{eq:S1/2def}
|f(x)| \lesssim e^{-\eta x^{2}}  \quad \mbox{and}\quad  |\widehat{f}(\xi)| \lesssim e^{- \eta \xi^{2}} 
\qquad \mbox {for some }\eta> 0.
\end{equation}
If we choose the Fourier transform as
$$
\widehat{f}(\xi) = (2\pi )^{-\frac 12}\int_{-\infty}^{\infty} f(t) e^{- i t \xi}\, dt,
$$
then according to Hardy's uncertainty principle \cite{Hardy33} (cf. \cite{FS1997,Jaming2006}),  we must have $0<\eta\leq1/2$
if \eqref{eq:S1/2def} holds for some non-identically zero function $f\in \mathcal{S}_{1/2}
(\mathbb{R})$. Furthermore, in the extreme case $\eta=1/2$, the function $f$ must necessarily be a constant 
multiple of the normalized Gauss function $\phi(x):= \pi^{-1/4} e^{-\frac{1}{2} x^{2}}$.

\par

In this article, we shall study spaces of functions for which the optimal time-frequency decay
in \eqref{eq:S1/2def} is almost reached. We first consider the Fr\'{e}chet space of functions
on $\R$ such that  	
\begin{equation}
\label{eq:H01/2TFDecay}
|f(x)| \lesssim e^{-(\frac{1}{2} - \lambda) x^{2}} 
\quad \mbox{and}\quad
|\widehat{f}(\xi)| \lesssim e^{-(\frac{1}{2} - \lambda) \xi^{2}} ,
\qquad \forall \lambda > 0 .
\end{equation}
Naturally, the elements of our space extend to entire functions on $\mathbb{C}$.
It is recently suggested in \cite[p.~4]{T-G-FourierCharPilipovicSp} that this Fr\'{e}chet space might coincide with the subspace of $L^{2}(\mathbb{R})$ consisting of functions such that their Hermite series coefficients are of rapid exponential decay. The first aim of this paper is to give an affirmative answer to this conjecture (see Theorem \ref{t:TFCharH01/2} below). Let 
	\[ h_{n}(x) = (-1)^{n} \pi^{-\frac{1}{4}} (2^{n} n!)^{-1/2} e^{\frac{1}{2} x^{2}} \frac{d^{n}}{dx^{n}} [e^{-x^{2}}] , \qquad n\in\N,\]
be the Hermite functions. They form an orthonormal basis of $L^{2}(\R)$ and consequently every $f\in L^{2}(\R)$ can be written as 
\begin{equation}
\label{eq:Hexpansion}
f = \sum_{n \in \N} H(f, n) h_{n}, \qquad \mbox{with } \  H(f, n) = \ev{f}{h_{n}}.
\end{equation}

	\begin{theorem}		\label{t:TFCharH01/2}
	 A function $f\in L^{2}(\mathbb{R})$ satisfies \eqref{eq:H01/2TFDecay} if and only if
			\begin{equation}
				\label{eq:HermiteCoeffH01/2}
				 |H(f, n)| \lesssim e^{-rn} , \qquad \forall r > 0 . 
			\end{equation}
	\end{theorem}
	
We observe that a function $f$ satisfies \eqref{eq:HermiteCoeffH01/2} 
if and only if $f$
belongs to the Pilipovi{\'c} space $\mathcal H _{0,1/2}(\mathbb R)$, see
\cite{T-ImagFuncDistSpBargmannTrans}. Hence, the previous
result is the same as the following.

	\begin{theorem}		\label{t:TFCharH01/2B}
	 Suppose $f\in L^{2}(\mathbb{R})$. Then $f\in \mathcal H _{0,1/2}(\mathbb R)$
	 if and only if \eqref{eq:H01/2TFDecay} holds true.
	\end{theorem}
Our second main result refines Theorem \ref{t:TFCharH01/2} and provides	
Hermite expansion characterizations for functions satisfying  
time-frequency decay estimates governed through suitable weight functions.
Here, by a \emph{weight function} we mean an unbounded non-decreasing function $\omega : [0, \infty) \to [0, \infty)$. 
See also Section \ref{sec:WeightFunc} for some explanation on the notions occurring in 
the next
statement.

	\begin{theorem}
		\label{t:FourierCharProperPilipovic}
		Let $\omega$ be a weight function such that $\varphi(t)=\omega(e^{t})$ is convex, and such that
	\begin{equation}
	\label{eq:keyHcharcond}
	 \int_{x}^{\infty}\frac{\omega(t)}{t^3}\, dt =O\left( \frac{\omega(x)}{x^2} \right), \qquad x\to\infty .
	\end{equation}
	Then, a function $f \in L^{2}(\R)$ satisfies
	\begin{equation}
				\label{eq:FourierCharProperPilipovic}
				|f(x)| \lesssim e^{-\frac{1}{2} x^{2} + \lambda \omega(|x|)}  \quad \mbox{and} \quad |\widehat{f}(\xi)| \lesssim e^{-\frac{1}{2} \xi^{2} + \lambda \omega(|\xi|)} 
		\end{equation}
		 for some $\lambda > 0$ (for every $\lambda > 0$)
		if and only if  its Hermite coefficients satisfy
			\begin{equation}
				\label{eq:HermiteCharProperPilipovic}
				|H(f, n)| \lesssim \sqrt{n!} \cdot e^{-\frac{1}{r} \varphi^{*}(r n)}
			\end{equation}
			for some $r > 0$ (for every $r > 0$), where $\varphi^{*}$ is the Young conjugate of $\varphi$.
	\end{theorem}

We point out that Theorem \ref{t:TFCharH01/2} itself will play an important role in showing Theorem \ref{t:FourierCharProperPilipovic}.

Evidently, if $\H _\omega (\R)$ ($\H _{0,\omega} (\R)$)
is the set of all $f\in L^2(\R)$ such that \eqref{eq:HermiteCharProperPilipovic}
holds true for some $r>0$ (for every $r>0$), then Theorem
\ref{t:FourierCharProperPilipovic} can be reformulated as follows.

	\begin{theorem}
		\label{t:FourierCharProperPilipovic2}
		Suppose $f \in L^{2}(\R)$, and let $\omega$ be a weight function such that $\varphi(t) = \omega(e^{t})$ is convex and that
		\eqref{eq:keyHcharcond} holds.
		Then $f\in \H _\omega (\R)$
		($f\in \H _{0,\omega} (\R)$) if and only if
		\eqref{eq:FourierCharProperPilipovic} holds true
		 for some $\lambda > 0$ (for every $\lambda > 0$).
		\end{theorem}

It turns out that with certain special choices of the weight functions (see Example \ref{ex:ClassicalPilipovicSp}) one can recover the full family of so-called proper Pilipovi\'{c} spaces,  introduced in \cite{T-ImagFuncDistSpBargmannTrans} and further studied in
\cite{AbFeGaToUs,FerGalTof,T-G-FourierCharPilipovicSp}.
Hence, in the one-dimensional case, our Theorem \ref{t:FourierCharProperPilipovic} significantly improves and extends upon  \cite[Theorem 2.2 and Theorem 2.3]{T-G-FourierCharPilipovicSp}, avoiding any need for the use of partial fractional Fourier transforms and, in particular, supplying so a transparent Fourier characterization of the Pilipovi\'{c} spaces. For example, in the terminology from
\cite{T-BargmannTransModGSSp,T-G-FourierCharPilipovicSp}, 
Theorem \ref{t:TFCharH01/2} is a Fourier characterization of the ``largest'' proper Pilipovi\'{c} space.
More generally, all these improvements are summarized by the next corollary; we refer to Section \ref{sec:FourierCharPilipovicSp} and Example \ref{ex:ClassicalPilipovicSp} below for unspecified notation.  
We just mention that for the sake of symbolic consistency, we have extended the definition of the family of Pilipovi\'{c} spaces to include $\mathcal{H}_{0,0}(\mathbb{R})$,  the one-dimensional space generated by the Gaussian $\phi$. 
	
	\begin{corollary}
		\label{t:TFCharClassicalPilipovicSp}
		Let $s \in \overline{\R}_{\flat}$ be such that $0 \leq s < 1/2$ (such that $0 \leq s \leq 1/2$).
		For any $f \in L^{2}(\R)$, then
		$f \in \H_{s}(\R)$ ($f \in \H_{0, s}(\R)$) if and only if
			\[ |f(x)| \lesssim e^{-\frac{1}{2} x^{2} + \lambda \omega_{s}(|x|)}  \quad \mbox{and}\quad|\widehat{f}(\xi)| \lesssim e^{-\frac{1}{2} \xi^{2} + \lambda \omega_{s}(|\xi|)} , \quad \text{for some } \lambda > 0 ~ (\text{for all } \lambda > 0) , \]
		where  the weight function $\omega_{s}$ is given by
			\begin{equation}
				\label{eq:WeightFuncClassicalPilipovicSp}
				\omega_{s}(t) =
					\begin{cases}
						\log (1 + t)^{\frac{1}{1 - 2s}} , & s \in \R , 0\leq s < 1/2 , \\
						t^{\frac{2 \sigma}{\sigma + 1}} , & s = \flat_{\sigma} , \\
						t^{2} , & s = 1/2 .
					\end{cases}
			\end{equation}
	\end{corollary}
	
Our article is structured as follows. In Section \ref{sec:WeightFunc} we briefly discuss some basic properties of weight functions. 
Section \ref{sec:PhragmenLindelof} provides a version of the Phragm\'{e}n-Lindel\"{o}f principle which allows one to deduce growth properties of analytic functions on sectors with respect to a weight function. This Phragm\'{e}n-Lindel\"{o}f result, when combined with the Bargmann and short-time Fourier transforms, is the key to our argument for proving Theorems \ref{t:TFCharH01/2} and \ref{t:FourierCharProperPilipovic}.  As a matter of fact, under certain assumptions on $\omega$, Theorem \ref{t:PhragmenLindelofChar} provides a complete characterization of the validity of the concrete form of the Phragm\'{e}n-Lindel\"{o}f principle that we need. The proofs of Theorems \ref{t:TFCharH01/2} and \ref{t:FourierCharProperPilipovic} are then given in Section \ref{sec:FourierCharPilipovicSp}.
Finally, in Section \ref{sec:MultDim}, we provide multidimensional analogs of our main results.

We would like to close this introduction by mentioning that there is an extensive literature on the interplay between Hermite expansions and function space theory, see e.g. \cite{BDJ2003, Langenbruch2006,Pilipovic88,Tan,Zhang63}. (See also \cite{VV2016} for eigenexpansions with respect to elliptic Shubin operators.)

\section{Preliminaries: weight functions}
\label{sec:WeightFunc}

Let $\omega$ be a weight function and let $\sigma > 0$. We will often impose one or more of the following conditions
 on $\omega$ at $\infty$ (we refer to \cite{B-M-T-UltradiffFuncFourierAnal,M-T-WhitneyExtensionThmUltradiffBeurling} for their analysis):

	\begin{itemize}
		\item[$(\alpha)$] $\omega(t + s) \leq L [ \omega(t) + \omega(s) + 1] ,$ for some $L\geq 1$ and $\forall t, s \geq 0$ ;
		\item[$(\beta_{\sigma})$] $ \int_{1}^{\infty} \omega(s)s^{-1 - \sigma}\, ds < \infty$; 
		\item[$(\beta^{\ast}_{\sigma})$] $\int_{1}^{\infty} \omega(ts) s^{-1 - \sigma}\, ds = O(\omega(t))$; 
		\item[$(\gamma)$] $\log t = o(\omega(t))$;
		\item[$(\delta)$] the function $\varphi(t) = \omega(e^{t})$ is convex.
	\end{itemize}	
Note that $(\beta_{2}^{\ast})$ and $(\delta)$ are precisely our hypotheses on $\omega$ in Theorem \ref{t:FourierCharProperPilipovic}. Throughout the article, $L$ always stands for the constant occurring in $(\alpha)$.
	
Two weight functions $\omega$ and $\kappa$ are called \emph{equivalent} if $\omega(t)=O(\kappa(t))$ and $\kappa(t)=O(\omega(t))$; as customary, one writes in such a case $\omega \asymp \kappa$. Observe that each of the conditions above, except for $(\delta)$, is invariant under equivalence of weight functions. We will make use of the following lemma, whose proof is given within the one of  \cite[Lemma 1.7]{B-M-T-UltradiffFuncFourierAnal}.
	
	\begin{lemma}
		\label{l:SmoothWeightFunc}
		For any weight function $\omega$ there exists a smooth weight function $\kappa$ such that $\omega \asymp \kappa$. Moreover, if $\omega$ satisfies $(\delta)$, $\kappa$ can be chosen to satisfy it as well.
	\end{lemma}

Let us now review several notions and results related to our conditions. We first consider the following well-known fact, which we will use in the sequel without mentioning it.
	\begin{lemma}
		If a weight function $\omega$ satisfies $(\beta^{\ast}_{\sigma})$ for $\sigma > 0$, then, $\omega$ satisfies $(\alpha)$ and $(\beta_{\sigma})$.
	\end{lemma}
	\begin{proof}
		Trivially $\omega$ satisfies $(\beta_{\sigma})$, while, since $\omega$ is non-decreasing, we have
			\[ \omega(2t) \leq \sigma 2^{\sigma} \int_{2}^{\infty} \omega(s t) s^{-1 - \sigma}\, ds = O(\omega(t)) ,  \]
		which yields $(\alpha)$.
	\end{proof}

We shall also need the next elementary lemma.
	\begin{lemma}
		\label{l:betasigma=>o(t^sigma)}
		If a weight function $\omega$ satisfies $(\beta_{\sigma})$ for $\sigma > 0$, then, $\omega(t) = o(t^{\sigma})$.
	\end{lemma}
	\begin{proof} 
		For $t>0$, since $\omega$ is non-decreasing, $t^{-\sigma} \omega(t) \leq \sigma \int_{t}^{\infty} \omega(s) s^{-\sigma-1}\, ds = o(1)$.
	\end{proof}	
	
Suppose that $\omega$ satisfies $(\delta)$. As in the statement of Theorem \ref{t:FourierCharProperPilipovic}, we will consider the \emph{Young conjugate $\varphi^{*}$} of the convex function $\varphi$, defined as usual as
	\[ \varphi^{*}(t) = \sup_{s > 0} st - \varphi(s), \qquad t\geq0 . \]	
We have that $\varphi^{*}$ is convex, unbounded, non-decreasing, and $(\varphi^{*})^{*} = \varphi$.  Note that, due to $(\delta)$, either $(\gamma)$ holds or $\omega\asymp \log^{+} := \max\{0, \log\}$.  When $\omega$ satisfies $(\gamma)$, $\varphi^{\ast}$ is finite-valued and $t=o(\varphi^{\ast}(t))$. If $(\gamma)$ does not hold, $\varphi^{\ast}(t)=\infty$ identically on some interval $(t_{0},\infty)$ and is finite valued on $[0,t_0]$.

\begin{remark}\label{rk:infphi}
In the latter case, we might always assume without loss of generality that our weight function is $\omega=\log^{+}$, so that $\varphi^{\ast}(t)=0$ for $t\leq1$ and $\varphi^{\ast}(t)=\infty$ for $t>1$. Even if $\varphi^{\ast}$ is infinite valued,  relations like \eqref{eq:HermiteCharProperPilipovic} have a clear meaning: If \eqref{eq:HermiteCharProperPilipovic} holds for some $r>0$, then $f$ is a finite linear combination of Hermite functions of degree $\leq1/r$; while if \eqref{eq:HermiteCharProperPilipovic} holds for all $r>0$, then $f$ is a scalar multiple of the Gaussian $\phi=h_0$.
\end{remark}

\section{An optimal Phragm\'{e}n-Lindel\"{o}f type result on sectors}
\label{sec:PhragmenLindelof}

Our primary tool in proving the announced results in the Introduction will be the application of the Phragm\'{e}n-Lindel\"{o}f principle on sectors. To this end, we shall establish in this section a weighted form of it and actually show that our result is optimal\footnote{As long as the weight functions satisfy $(\alpha)$ and $(\beta_{\sigma})$.}.
For any $\theta \in [0, 2 \pi)$ and $\varrho>0$ we consider the sector\footnote{If $\varrho\geq2\pi$, it is a subset of the Riemann surface of the logarithm, where its boundary for Theorem \ref{t:PhragmenLindelofChar} should be taken. }
	\[ S_{\theta, \varrho} = \{ z \in \C \mid \theta - \varrho / 2 < \arg z < \theta + \varrho / 2 \} . \]

The main result of this section may then be stated as follows.

	\begin{theorem}
		\label{t:PhragmenLindelofChar}
		Let $\sigma > 0$ and let $\omega$ be a weight function satisfying $(\alpha)$ and $(\beta_{\sigma})$. Then the following statements are equivalent:
			\begin{itemize}
				\item[(I)] $\omega$ satisfies $(\beta_{\sigma}^{\ast})$.
				\item[(II)] There is a constant $A > 0$ such that for any $\theta \in [0, 2 \pi)$, $\varrho \in (0, \pi / \sigma]$, and $\lambda>0$ the following holds: there is a $B>0$ such that
				 if $F$ is an analytic function on the sector $S_{\theta, \varrho}$ with continuous extension to $\overline{S}_{\theta, \varrho}$ for which 
					\[ |F(z)| \leq M e^{\lambda \omega(|z|)} , \qquad z \in \partial S_{\theta, \varrho} , \]
				for some $M=M_\lambda >0$, and
					\[ |F(z)| \lesssim_{\varepsilon} e^{\varepsilon |z|^{\sigma}} , \qquad z \in S_{\theta, \varrho} , \]
				for all $\varepsilon > 0$, then
					\[ |F(z)| \leq B M e^{A \lambda \omega(|z|)} , \qquad  z \in S_{\theta, \varrho} . \]
			\end{itemize}
Furthermore, the constants $A$ and $B$ in \textnormal{(II)} can be chosen to be absolute, in the sense that $A$ only depends on $\omega$ and $\sigma$, while $B$ only depends on $\omega$, $\sigma$, and $\lambda$.
			
	\end{theorem}
	
For our proof of Theorem \ref{t:PhragmenLindelofChar}, we shall need two auxiliary lemmas. The first one is the following well-known classical version of the 
Phragm\'{e}n-Lindel\"{o}f principle on sectors.
	
	\begin{lemma}[{\cite[Theorem 1.4.3]{BoasBook}}]
		\label{l:PhragmenLindelofClassic} 
		Let $\sigma>0$, $\theta \in [0, 2 \pi)$, and $\varrho \in (0, \pi / \sigma]$. Suppose $F$ is an analytic function on $S_{\theta, \varrho}$ with continuous extension to $\overline{S}_{\theta, \varrho}$ such that, for some $M > 0$,
			\[ |F(z)| \leq M , \qquad z \in \partial S_{\theta, \varrho} , \]
		and, for every $\varepsilon > 0$,
			\[ |F(z)| \lesssim_{\varepsilon} e^{\varepsilon |z|^{\sigma}} , \qquad z \in S_{\theta, \varrho} . \]
		Then,
			\[ |F(z)| \leq M , \qquad  z \in \overline{S}_{\theta, \varrho} . \]
	\end{lemma}
	
Next, we need the construction of a useful class of analytic functions on sectors with certain desirable growth properties with respect to the weight function $\omega$.

	\begin{lemma}
		\label{l:SuitableAnalyticFunc}
		Let $\sigma > 0$. Let $\omega$ be a weight function satisfying $(\alpha)$ and $(\beta_{\sigma})$. For any $\theta \in [0, 2\pi)$ there is an analytic function $F_{\omega}$ on $S_{\theta, \pi / \sigma}$ with continuous extension to $\overline{S}_{\theta, \pi / \sigma}$ such that for certain $c, \nu_{1}, \mu > 0$ it holds
			\begin{itemize}
				\item[(i)] for every $\varepsilon > 0$ there is a $C_{\varepsilon} > 0$ such that
					\[ c e^{\nu_{1} \omega(|z|)} \leq |F_{\omega}(z)| \leq C_{\varepsilon} e^{\varepsilon |z|^{\sigma}} , \qquad z \in S_{\theta, \pi / \sigma} ; \]
				\item[(ii)] $|F_{\omega}(z)| \lesssim e^{\mu \omega(|z|)}$ for $z \in \partial S_{\theta, \pi / \sigma}$;
				\item[(iii)] $\int_{0}^{\infty} \frac{\omega(ts)}{s^{1-\sigma} + s^{\sigma + 1}}\, ds = O\left(\log |F_{\omega}(t e^{i \theta})|\right)$;
				\item[(iv)] if in addition $\omega$ satisfies $(\beta_{\sigma}^{\ast})$, there are $C, \nu_{2} > 0$ such that $|F_{\omega}(z)| \leq C e^{\nu_{2} \omega(|z|)}$ for $z \in S_{\theta, \pi / \sigma}$.
				
			\end{itemize}
	\end{lemma}
	
	\begin{proof}
		It suffices to consider the case where $\sigma = 1$ and $\theta = 0$. Indeed, for the general case we set $\widetilde{\omega}(t) = \omega(t^{1/\sigma})$ and let $G = F_{\widetilde{\omega}}$ be the function corresponding to $\sigma = 1$ and $\theta = 0$; we then obtain that $F_{\omega}(z) = G(z^{\sigma} e^{- i \sigma \theta})$, for $z \in \overline{S}_{\theta, \pi / \sigma}$, satisfies the desired properties.
		
		We therefore assume that $\sigma = 1$ and $\theta = 0$.
		Using Lemma \ref{l:SmoothWeightFunc}, there exists a weight function $\kappa \in C^{1}(\R)$ such that $\kappa \asymp \omega$. 
		Hence, for some $c_{\kappa}, A_{\kappa}, C_{\kappa} > 0$ we have that $c_{\kappa} \omega(t) - A_{\kappa} \leq \kappa(t) \leq C_{\kappa} (\omega(t) + 1)$ for all $t \in \R$.
		Put
			\[ P_{\kappa}(x, y) = \frac{x}{\pi} \int_{-\infty}^{\infty} \frac{\kappa(t)}{x^{2} + (y - t)^{2}}\, dt . \]
		By \cite[Lemma 2.2]{B-M-T-UltradiffFuncFourierAnal} and Lemma \ref{l:betasigma=>o(t^sigma)}
		the harmonic function $P_{\kappa}$
		is such that for every $\varepsilon > 0$ there is some $C_{\varepsilon} > 0$ for which
			\[  \frac{1}{4} \kappa(|z|) \leq P_{\kappa}(x,y) \leq \varepsilon |z| + C_{\varepsilon} , \qquad z=x+iy \in S_{0, \pi } , \]
		and $P_{\kappa}(x, 0) = \pi^{-1} \int_{-\infty}^{\infty}(1 + t^{2})^{-1}  \kappa(x t) dt$ for $x > 0$.
		In particular, it follows that
			\[ \frac{c_{\kappa}}{4} \omega(|z|) - \frac{A_{\kappa}}{4} \leq P_{\kappa}(x, y) \leq \varepsilon |z| + C_{\varepsilon} , \qquad z=x+iy \in S_{0, \pi} , \]
		and
			\[ \frac{2 c_{\kappa}}{\pi} \int_{0}^{\infty} \frac{\omega(x t)}{1 + t^{2}} dt \leq P_{\kappa}(x, 0) + A_{\kappa} , \qquad x > 0 .  \]
		If $(\beta_{1}^{\ast})$ holds, it follows that
			\begin{align*} 
				P_{\kappa}(x, y) - C_{\kappa} L \omega(y) 
				&= \frac{x}{\pi} \int_{-\infty}^{\infty} \frac{\kappa(t) - C_{\kappa} L \omega(y)}{x^{2} + (y - t)^{2}} dt
				\leq C_{\kappa} \frac{x}{\pi} \int_{-\infty}^{\infty} \frac{\omega(t) - L \omega(y)}{x^{2} + (y - t)^{2}} dt + C_{\kappa} \\
				&\leq \frac{C_{\kappa} L}{\pi} \int_{-\infty}^{\infty} \frac{\omega(x s)}{1 + s^{2}}\, ds + C_{\kappa} + L \leq R \omega(x) + C_{\kappa} + L , 
			\end{align*}
		for some $R \geq C_{\kappa} L$.
		Hence, $P_{\kappa}(x, y) \leq 2 R \omega(|z|) + C_{\kappa} + L$. 
		As $\kappa$ is differentiable it follows \cite[p.~109]{K-IntroHpSp} that $\lim_{x \to 0^{+}} P_{\kappa}(x, y) = \kappa(y)$ for any $y \in \R$, and we may thus define $P_{\kappa}(0, y) = \kappa(y)$ extending it to a continuous function on $\overline{S}_{0, \pi}$. Let $V_{\kappa} : S_{0, \pi} \to \R$ be a harmonic conjugate for $P_{\kappa}$ in $S_{0, \pi}$. 
		Integration by parts shows that $\kappa'$ satisfies $(\beta_{1})$, hence $\partial P_{\kappa}/\partial y = P_{\kappa'}$ has continuous extension to $\overline{S}_{0, \pi}$. Then, by the Cauchy-Riemann equations,
			\begin{align*} 
				\lim_{x \to 0^{+}} V_{\kappa}(x, y) 
				&= V_{\kappa}(1, y) - \lim_{x \to 0^{+}} \int_{x}^{1} \frac{\partial V_{\kappa}}{\partial x}(u, y)\, du 
				= V_{\kappa}(1, y) + \lim_{x \to 0^{+}} \int_{x}^{1} \frac{\partial P_{\kappa}}{\partial y}(u, y)\, du \\
				&= V_{\kappa}(1, y) + \int_{0}^{1} \frac{\partial P_{\kappa}}{\partial y}(u, y)\, du . 
			\end{align*}
		Consequently, we may extend $V_{\kappa}$ to a continuous function on $\overline{S}_{0, \pi}$. Finally, we set $F_{\omega}(z) = \exp[P_{\kappa}(x, y) + i V_{\kappa}(x, y)]$.
		Then the desired properties (i)--(iii) and (iv) if $\omega$ satisfies $(\beta_{1}^{\ast})$ follow for $F_{\omega}$ by setting $c = \exp[-A_{\kappa}/4]$, $C = \exp[C_{\kappa} + L]$, $\nu_{1} = c_{\kappa} / 4$, $\nu_{2} = 2 R$, and $\mu = C_{\kappa}$.
	\end{proof}
	
	\begin{proof}[Proof of Theorem \ref{t:PhragmenLindelofChar}]		
		(I) $\Rightarrow$ (II). Let $\theta, \varrho, F$, $\lambda$ be as in (II). Let $F_{\omega}$ be as in Lemma \ref{l:SuitableAnalyticFunc} corresponding to $\sigma$ and $\theta$. Consider $G_{\lambda}(z) = F(z) F_{\omega}(z)^{-\lambda / \nu_{1}}$. Using the lower bound from Lemma \ref{l:SuitableAnalyticFunc}(i), $|G_{\lambda}(z)| \leq c^{-\lambda / \nu_{1}} M$ on $\partial S_{\theta, \varrho}$. Also, $|G_{\lambda}(z)| \lesssim e^{\varepsilon |z|^{\sigma}}$ for any $\varepsilon > 0$. Applying Lemma \ref{l:PhragmenLindelofClassic}, we get that $|F(z)| \leq c^{-\lambda / \nu_{1}} M |F_{\omega}(z)|^{\lambda / \nu_{1}}$. The result follows from Lemma \ref{l:SuitableAnalyticFunc}(iv) with $A=\nu_{2} / \nu_{1}$ and $B= (C / c)^{\lambda / \nu_{1}}$.
		
		(II) $\Rightarrow$ (I). Let $F=F_{\omega}$ be the function as in Lemma \ref{l:SuitableAnalyticFunc}.
		If $(\beta_{\sigma}^{\ast})$ does not hold, we obtain from 	Lemma \ref{l:SuitableAnalyticFunc}(iii),
\[ \limsup_{t\to\infty} |F(te^{i\theta})|e^{-\lambda \omega(t)}=\infty
\]
for each $\lambda>0$,  proving that the statement (II) does not follow for this $F$.
	\end{proof}

\section{Proofs of the main theorems}
\label{sec:FourierCharPilipovicSp}

In this section, we shall always assume that the weight function $\omega$ satisfies $(\delta)$ so that $\varphi^{*}$ is well-defined (cf. Remark \ref{rk:infphi}). Our proofs of the main theorems stated in the Introduction will be preceded by some intermediate results, and we need to introduce some notions in preparation.

As in \cite{T-ImagFuncDistSpBargmannTrans}, let 
	\[ \H_{0}(\R) = \text{span } \{ h_{n} \mid n \in \N \} , \]
and let $\H^{\prime}_{0}(\R)$ be its dual space. In particular, we may always write $f \in \H'_{0}(\R)$ as in \eqref{eq:Hexpansion}, where now $H(f, n) = \ev{f}{h_{n}}$ stands for dual pairing evaluation. One should regard $\H'_{0}(\R)$ as a large ultradistribution space that is Fourier invariant. Indeed, since 
$\widehat{h}_{n}=(-i)^{n}h_n$, the Fourier transform can naturally be extended to an automorphism on $\H'_{0}(\R)$. 
 
We now introduce spaces via the bounds \eqref{eq:HermiteCharProperPilipovic}. For any $r > 0$, we define the Banach space
	\[ \H_{\omega, r}(\R) = \{ f \in \H^{\prime}_{0}(\R) \mid \|f\|_{\H_{\omega, r}} < \infty \},
	\quad \|f\|_{\H_{\omega, r}} := \sup_{n \in \N} \left (
	\frac{|H(f, n)|}{\sqrt{n!}} e^{\frac{1}{r} \varphi^{*}(rn)} \right ). \]
Then, we consider the two spaces
	\[ \H_{\omega}(\R) = \bigcup_{r > 0} \H_{\omega, r}(\R) , \qquad \H_{0, \omega}(\R) = \bigcap_{r > 0} \H_{\omega, r}(\R) , \]
endowed with their natural $(LB)$-space and  Fr\'{e}chet space topologies, respectively.
Note that $\H_{\omega}(\R) = \H_{\kappa}(\R)$ and $\H_{0, \omega}(\R) = \H_{0, \kappa}(\R)$ whenever $\omega \asymp \kappa$ (and if $\kappa$ also satisfies $(\delta)$). 

If $\omega\asymp \log^{+}$, then $\mathcal{H}_{\log^{+}}(\mathbb{R})=\mathcal{H}_{0}(\mathbb{R})$, while $\mathcal{H}_{0,\log^{+}}(\mathbb{R})=\{c\phi \:| \: c\in\mathbb{C}\}$, with $\phi$ being the standard Gaussian. 
When $\omega$ satisfies $(\gamma)$ and $\omega(t)=o(t^{2})$, we have the following continuous dense inclusions:
	\[ \H_{0}(\R) \subsetneq \H_{0, \omega}(\R) \subsetneq \H_{\omega}(\R) \subsetneq \S_{1/2}(\R) \subsetneq \S(\R), \]
where the inclusion into  $\S_{1/2}(\R)$ follows from Zhang's classical characterization of the Gelfand-Shilov spaces \cite{Zhang63}.

The classical Pilipovi\'{c} spaces are given in the following example.

	\begin{example}
		\label{ex:ClassicalPilipovicSp}
		Consider the set $\overline{\R_{\flat}} = [0, \infty) \cup \{ \flat_{\sigma} \}_{\sigma > 0}$, where besides the usual ordering in $\R$, the elements $\flat_{\sigma}$ in $\overline{\R_{\flat}}$ are ordered by the relations $x_{1} < \flat_{\sigma_{1}} < \flat_{\sigma_{2}} < x_{2}$, when $\sigma_{1}, \sigma_{2}, x_{1}$, and $x_{2}$ are positive real numbers such that $\sigma_{1} < \sigma_{2}$, $x_{1} < 1/2$, and $x_{2} \geq 1/2$. For $s \in \overline{\R_{\flat}}$ with $s \leq 1/2$, let $\omega_{s}$ be as in \eqref{eq:WeightFuncClassicalPilipovicSp}. Then, $\omega_{s}$ is a weight function satisfying $(\alpha)$, $(\gamma)$ (except if $s = 0$), and $(\delta)$. Moreover, $\omega_{s}$ satisfies $(\beta_{2}^{\ast})$ if and only if $s < 1/2$. Put $\varphi_{s}(t) = \omega_{s}(e^{t})$ and let $\varphi_{s}^{*}$ be its Young conjugate. Then a straightforward computation shows that
			\[ 
				\sqrt{n!} e^{- \frac{1}{r} \varphi^{*}_{s}(rn)} \asymp
					\begin{cases}
						e^{- r n^{\frac{1}{2s}}} , & s \in \R, 0 < s \leq 1/2 , \\
						r^{n} (n!)^{- \frac{1}{2 \sigma}} , & s = \flat_{\sigma} .
					\end{cases}
			 \]
		In particular, for $s \in \overline{\R_{\flat}}$ and $s \leq 1/2$, we have that $\H_{\omega_{s}}(\R) = \H_{s}(\R)$ if $s < 1/2$, and $\H_{0, \omega_{s}} = \H_{0, s}(\R)$ if $s > 0$, where $\H_{s}(\R)$ and $\H_{0, s}(\R)$ are the spaces as introduced in \cite{T-ImagFuncDistSpBargmannTrans}. 
	\end{example}

We may now focus on the proofs of Theorems \ref{t:TFCharH01/2} and \ref{t:FourierCharProperPilipovic}. Note that Corollary \ref{t:TFCharClassicalPilipovicSp} would then automatically follow from Example \ref{ex:ClassicalPilipovicSp}.
We start with the following observation that the Hermite coefficients are summable with respect to the weights.

	\begin{lemma}
		\label{l:HermiteCoeffSummable}
		$f \in \H_{0}^{\prime}(\R)$ belongs to $\H_{\omega}(\R)$ (resp. $\H_{0, \omega}(\R)$) if and only if for some $r > 0$ (for every $r > 0$):
			\[ \sum_{n \in \N} \frac{H(f, n)}{\sqrt{n!}} e^{\frac{1}{r} \varphi^{*}(r n)} < \infty . \] 
	\end{lemma}
	
	\begin{proof} According to our convention (cf. Remark \ref{rk:infphi}), this holds if $\omega\asymp \log^{+}$. If $\omega$ satisfies $(\gamma)$, this follows from  \cite[Proposition 2.1, Lemma 3.2, and Lemma 3.5(c)]{D-N-V-NuclGSSpKernThm}.
	\end{proof}

Our main tool will be the use of the Bargmann transform \cite{B-HilbertSpAnalFuncAssocIntTrans}, allowing us to work with entire functions. On $\H_{0}^{\prime}(\R)$ we define the \emph{Bargmann transform} as
	\[ \B : \H_{0}'(\R) \to \C[[z]] ,
	\quad f \mapsto \sum_{n \in \N} \frac{H(f, n)}{\sqrt{n!}} z^{n} , \]
where $\C[[z]]$ denotes the ring of formal power series over $\C$ (which is denoted by
$\maclA _0'(\mathbb C)$ in \cite{T-ImagFuncDistSpBargmannTrans}).

For any $f \in L^{2}(\R)$ we then have that $\B f$ is an entire function with
	\[ \B f(z) = \pi^{-\frac{1}{4}} \int_{-\infty}^{\infty} f(t) e^{-\frac{1}{2} (z^{2} + t^{2}) + \sqrt{2} t z}\, dt , \qquad z\in\mathbb{C}. \]
In particular, for the Hermite functions we have
	\begin{equation}
		\label{eq:BargmannHermite}
		\B h_{n}(z) = \frac{z^{n}}{\sqrt{n!}} , \qquad n \in \N . 
	\end{equation} 
It is now straightforward to characterize our spaces defined via Hermite expansions by using the Bargmann transform.
	
	\begin{proposition}
		\label{p:PilipovicBargmanChar}
		Let $\omega$ be a weight function satisfying  $(\delta)$. Then, for an entire function $F$ there exists an $f \in \H_{\omega}(\R)$ (resp. $f \in \H_{0, \omega}(\R)$) such that $F = \B f$ if and only if for some $\lambda > 0$ (for every $\lambda > 0$):
			\[ |F(z)| \lesssim e^{\lambda \omega(|z|)} , \qquad  z \in \C . \]
	\end{proposition}
	
	\begin{proof} This is clearly true if $\omega=\log^{+}$, since $\B(\H_{0}(\mathbb{R}))$ is just the space of polynomials on $\C$. We may thus assume that $(\gamma)$ holds.
		Suppose first that the estimate from Lemma \ref{l:HermiteCoeffSummable} holds for a given $r>0$. Put $\lambda = 1 / r$. By \eqref{eq:BargmannHermite} we obtain, for $|z| > 1$,
			\begin{align*}|
				\B f(z)| 
				&= |\B\left[ \sum_{n \in \N} H(f, n) h_{n} \right](z)|  \\
				&\leq \sum_{n \in \N} \frac{|H(f, n)|}{\sqrt{n!}} |z|^{n}  
				\lesssim \sup_{n \in \N} \left( \exp[ n \log |z| - \frac{1}{r} \varphi^{*}(r n) ] \right) \\
				&\leq e^{\lambda \omega(|z|)} .
			\end{align*}
		Conversely, suppose that $F$ is an entire function such that $|F(z)| \lesssim e^{\lambda \omega(|z|)}$ for some $\lambda > 0$. By the Cauchy inequalities, we have, for $r = 1 / \lambda$ and for every $n \in \N$,
			\[ |F^{(n)}(0)| \lesssim n! \inf_{t > 1} \left( \sup_{|z| = t} \frac{|F(z)|}{t^{n}} \right) \lesssim n! \inf_{t > 1} \left( \exp[ - (n \log t - \lambda \omega(t))] \right) = n! e^{- \frac{1}{r} \varphi^{*}(r n)} . \]
		Putting $c_{n} = |F^{(n)}(0)| / \sqrt{n!}$, then it follows by \eqref{eq:BargmannHermite} that $f = \sum_{n \in \N} c_{n} h_{n}$ is an element of $\H_{\omega, r}(\R)$ and $F = \B f$.
	\end{proof}
	
Next, we consider the \emph{short-time Fourier transform} \cite{GrochenigBook} with window $\phi(x)= \pi^{-1/4} e^{-\frac{1}{2} x^{2}}$. 
For any $f \in L^{2}(\R)$ it is defined as
	\[ V_{\phi} f(x, \xi) = (2\pi )^{-\frac 12} \int_{-\infty}^{\infty} f(t) \phi(t - x) e^{- i t \xi}\, dt , \qquad (x, \xi) \in \R^{2} . \]		The connection between $V_{\phi} f$ and $\B f$ is given by
	\begin{equation}
		\label{eq:STFTBargmann}
		V_{\phi} f(x, \xi) = (2\pi )^{-\frac 12} e^{-\frac{1}{4} (x^{2} + \xi^{2})} e^{-\frac{i}{2} x \xi} \B f\left(\frac{\overline{z}}{\sqrt{2}}\right) , \qquad  z = x + i \xi \in \C .  	\end{equation}
Using this, we now find the following two results concerning the time-frequency decay of a function and its bounds for the Bargmann transform.

	\begin{lemma}
		\label{l:TFDecayToBargmann}
		Let $\omega$ be a weight function satisfying $(\alpha)$ such that $\omega(t) = o(t^{2})$ (such that $\omega(t) = O(t^{2})$). If $f \in \H_{0}^{\prime}(\R)$ satisfies
			\[ |f(x)| \lesssim e^{-\frac{1}{2} x^{2} + \lambda \omega(|x|)} , \qquad |\widehat{f}(\xi)| \lesssim e^{-\frac{1}{2} \xi^{2} + \lambda \omega(|\xi|)} , \]
		for some $\lambda > 0$ (for small enough $\lambda > 0$), then,
			\[ |\B f(z)| \lesssim e^{\frac{1}{2} x^{2} + 2 L^{2} \lambda \omega(|\xi |)} ,
			\qquad |\B f(z)| \lesssim e^{\frac{1}{2} \xi ^{2} + 2 L^{2} \lambda \omega(|x|)},
			\qquad   z = x + i\xi  \in \C .
			 \]
	\end{lemma}
	
	\begin{proof}
		We have
			\begin{align*}
				|V_{\phi} f(x, \xi)| e^{\frac{1}{4} x^{2} - L \lambda \omega(|x|)}
				&\lesssim \int_{-\infty}^{\infty} |f(t)|
				e^{-\frac{1}{2} |t - x|^{2} + \frac{1}{4} x^{2}}
				e^{- L \lambda \omega(|x|)}\, dt
				 \\
				&= \int_{-\infty}^{\infty} (|f(t)| e^{\frac{1}{2} t^{2}})
				e^{-\frac{1}{4} |2t - x|^{2} - L \lambda \omega(|x|)}\,  dt
				\\
				&\lesssim \int_{-\infty}^{\infty} (|f(t)|
				e^{\frac{1}{2} t^{2} - \lambda \omega(|2t|)})
				e^{L \lambda \omega(|2t - x|) - \frac{1}{4} |2t - x|^{2}}\, dt
				\lesssim 1 ,
				\end{align*}
		where in the case $\omega(t)=O(t^2)$ one needs to choose $\lambda$ small enough so that the function $t \mapsto e^{L \lambda \omega(|2t - x|) - \frac{1}{4} |2t - x|^{2}}$ is integrable.
		Using \eqref{eq:STFTBargmann} we then find that $|\B f(z)| \lesssim e^{\frac{1}{2} \xi ^{2} + 2 L^2 \lambda \omega(|x|)}$. By observing that $V_{\phi} f(x, \xi) = e^{- i x \xi} V_{\phi} \widehat{f}(\xi, -x)$,  
		the other bound follows analogously.
	\end{proof}
	
		\begin{lemma}
		\label{l:BargmannToTFDecay}
		Let $\omega$ be a weight function satisfying $(\alpha)$ such that $\omega(t) = o(t^{2})$ (such that $\omega(t) = O(t^{2})$). For any entire function $F$ satisfying, for some  $\lambda > 0$ (small enough $\lambda >0$) ,
			\[ |F(z)| \lesssim  e^{\lambda \omega(|z|)}, \]
		there exists an $f \in \H_{0}^{\prime}(\R)$ such that
			\[ |f(x)| \lesssim e^{-\frac{1}{2} x^{2} + 2 L^{2} \lambda \omega(|x|)} , \qquad |\widehat{f}(\xi)| \lesssim e^{-\frac{1}{2} \xi^{2} + 2 L^{2} \lambda \omega(|\xi|)} , \]
		and $F = \B f$.
	\end{lemma}

	\begin{proof}
		As $F$ is an entire function with $|F(z)| \lesssim e^{c' |z|^{2}}$ for some $c' < 1/2$, there exists \cite{B-HilbertSpAnalFuncAssocIntTrans} $f \in L^{2}(\R)$ such that $F = \B f$. By the Fourier inversion formula, we have
			\[ \pi^{-\frac{1}{4}} e^{- \frac{1}{8} x^{2}} f(x / 2) =  f(x / 2) \phi(x / 2 - x) = \int_{-\infty}^{\infty} V_{\phi} f(x, \xi) e^{\frac{i}{2} x \xi} d\xi . \] 
		Via \eqref{eq:STFTBargmann} we get
			\[ \int_{-\infty}^{\infty} |V_{\phi} f(x, \xi)| d\xi \lesssim e^{-\frac{1}{4} x^{2} + L \lambda \omega(|x|)} \int_{-\infty}^{\infty} e^{L \lambda \omega(|\xi|) - \frac{1}{4} \xi^{2}} d\xi \lesssim e^{-\frac{1}{4} x^{2} + L \lambda \omega(|x|)} , \]
		where in the case $\omega(t)=O(t^2)$ one needs to choose $\lambda$ small enough such that the function $\xi \mapsto e^{L \lambda \omega(|\xi|) - \frac{1}{4} \xi^{2}}$ is integrable.	Consequently, $|f(x)| \lesssim e^{-\frac{1}{2} x^{2} + 2 L^{2} \lambda \omega(|x|)}$. 
		The other bound follows similarly by observing again that $V_{\phi} f(x, \xi) = e^{- i x \xi} V_{\phi} \widehat{f}(\xi, -x)$. 
	\end{proof}

We are ready to show our main results.

	\begin{proof}[Proof of Theorem \ref{t:TFCharH01/2}]
		Note that \eqref{eq:HermiteCoeffH01/2} exactly means  that $f$ is an element of $\H_{0, 1/2}(\R)$. By Proposition \ref{p:PilipovicBargmanChar} and Lemma \ref{l:BargmannToTFDecay} it follows that each element $f \in \H_{0, 1/2}(\R)$ satisfies \eqref{eq:H01/2TFDecay}.
		Suppose now that $f \in L^2(\R)$ is such that \eqref{eq:H01/2TFDecay} holds. By Lemma \ref{l:TFDecayToBargmann} we have
			\[ |\B f(z)| \lesssim e^{\frac{1}{2} x^{2} + \lambda y^{2}} , \qquad |\B f(z)| \lesssim e^{\frac{1}{2} y^{2} +  \lambda x^{2}} , \qquad \forall \lambda > 0 . \]
		We will show that $|\B f(z)| \lesssim e^{\lambda |z|^{2}}$ for all $\lambda > 0$ in $S_{\pi / 4, \pi / 2}$, the same bounds on the other quadrants will analogously follow. For any $\varepsilon \in (0, 1)$ let $\theta = \arctan (1/(4 \varepsilon))$. Put $G_{\varepsilon}(z) = \B f(z) e^{-\varepsilon (z e^{-i \frac{\pi}{4}})^{2}}$. Then, for every $\lambda > 0$, $|G_{\varepsilon}(x)| \lesssim e^{\lambda x^{2}}$ for $x \geq 0$, while for $z = \rho e^{i \theta}$, $\rho > 0$, we have
			\[ |G_{\varepsilon}(z)| = |\B f(z)| e^{-\varepsilon \cos\left(2 \theta - \frac{\pi}{2}\right) |z|^{2}} = |\B f(z)| e^{-2 \varepsilon \tan \theta \cdot x^{2}} \lesssim e^{\lambda y^{2}} .  \]
		By Theorem \ref{t:PhragmenLindelofChar} we get that $|G_{\varepsilon}(z)| \lesssim e^{\lambda |z|^{2}}$ for all $\lambda > 0$, so that in particular $|\B f(z)| \lesssim e^{2\varepsilon |z|^{2}}$, for all $z \in \C$ with $0 \leq \arg z \leq \theta$. Moreover, for any $\varrho \in (\theta, \pi / 2]$ we have
			\[ |\B f(\rho e^{i \varrho})| \lesssim e^{\frac{1}{2} \cos^{2} \theta \cdot \rho^{2} + \varepsilon \rho^{2}} = e^{8 \varepsilon^{2} \sin^{2} \theta \cdot \rho^{2} + \varepsilon \rho^{2}}\leq  e^{9 \varepsilon \rho^{2}} . \]
		Consequently, we have  $|\B f(z)| \lesssim e^{\lambda |z|^{2}}$ for every $\lambda > 0$ in $S_{\pi / 4, \pi / 2}$. Our proof is then completed by Proposition \ref{p:PilipovicBargmanChar}.
	\end{proof}

	\begin{proof}[Proof of Theorem \ref{t:FourierCharProperPilipovic}]
		By definition we have that \eqref{eq:HermiteCharProperPilipovic} holds if and only if $f \in \H_{\omega}(\R)$ ($f \in \H_{0, \omega}(\R)$).
		If the latter is true, then \eqref{eq:FourierCharProperPilipovic} follows from Proposition \ref{p:PilipovicBargmanChar} and Lemma \ref{l:BargmannToTFDecay}.
		Conversely, if $f$ satisfies \eqref{eq:FourierCharProperPilipovic}, then by Lemma \ref{l:TFDecayToBargmann} we have that $|\B f(x)| \lesssim e^{\lambda \omega(|x|)}$ and $|\B f(i y)| \lesssim e^{\lambda \omega(|y|)}$ for every $x, y \in \R$ and some $\lambda > 0$ (all $\lambda > 0$).
		Also, since $\omega(t) = o(t^{2})$ by Lemma \ref{l:betasigma=>o(t^sigma)}, we infer from Theorem \ref{t:TFCharH01/2} and Proposition \ref{p:PilipovicBargmanChar} that $|\B f(z)| \lesssim e^{\varepsilon |z|^{2}}$ for every $\varepsilon > 0$.
		Theorem \ref{t:PhragmenLindelofChar} then implies that $|\B f(z)| \lesssim e^{\lambda \omega(|z|)}$ for some $\lambda > 0$ (for every $\lambda > 0$). Consequently, $f \in \H_{\omega}(\R)$ ($f \in \H_{0, \omega}(\R)$) as another application of Proposition \ref{p:PilipovicBargmanChar} yields.  
	\end{proof}	
	
\section{The multidimensional case}
\label{sec:MultDim}

In this final section, we review multidimensional analogs of our main results.
Their proofs are similar to the one-dimensional case and are therefore omitted.
Throughout this section, we fix $d \geq 1$ and it denotes the dimension.
 
We start by introducing the multidimensional counterparts of the spaces introduced in Section \ref{sec:FourierCharPilipovicSp}.
Let
	\[ h_{\alpha}(x) = (-1)^{|\alpha|} \pi^{-\frac{d}{4}} (2^{|\alpha|} \alpha!)^{-1/2} e^{\frac{1}{2} |x|^{2}} \frac{\partial^{\alpha}}{\partial x^{\alpha}} [e^{-|x|^{2}}] , \qquad \alpha \in \N^{d} , \]
denote the $d$-dimensional Hermite functions. 
Put $\H_{0}(\R^{d}) = \spn \{ h_{\alpha} \mid \alpha \in \N^{d} \}$ and let $\H^{\prime}_{0}(\R^{d})$ be its dual space.
For any $f \in \H^{\prime}_{0}(\R^{d})$ we write $H(f, \alpha) = \ev{f}{h_{\alpha}}$, $\alpha \in \N^{d}$, so that $f = \sum_{\alpha \in \N^{d}} H(f, \alpha) h_{\alpha}$.
Then, for a weight function $\omega$ satisfying $(\delta)$ and $r > 0$ we consider the Banach space
	\[ \H_{\omega, r}(\R^{d}) = \{ f \in \H^{\prime}_{0}(\R^{d}) \mid \|f\|_{\H_{\omega, r}} < \infty \}  , \quad \|f\|_{\H_{\omega, r}} := \sup_{\alpha \in \N^{d}} \left( \frac{|H(f, \alpha)|}{\sqrt{\alpha!}} e^{\frac{1}{r} \varphi^{*}(r |\alpha|)} \right) .   \]
Finally, we define our two spaces of interest:
	\[ \H_{\omega}(\R^{d}) = \bigcup_{r > 0} \H_{\omega, r}(\R^{d}) , \qquad \H_{0, \omega}(\R^{d}) = \bigcap_{r > 0} \H_{\omega, r}(\R^{d}) . \]
	
As before, our aim is to provide Fourier characterizations of $\H_{\omega}(\R^{d})$ and $\H_{0, \omega}(\R^{d})$ for suitable weight functions $\omega$.
However, the techniques used for the one-dimensional case appear not to
 translate directly to multiple dimensions. 
More specifically, for our characterization, it will not suffice only considering the decay of the function and its Fourier transform.
Instead, we also look at tensor products of partial Fourier transforms:
	\[ \F_{\varpi} = \bigotimes_{j = 1}^{d} \F^{\varpi_{j}} , \qquad \varpi \in \{0, 1\}^{d} , \]
where $\F$ denotes the one-dimensional Fourier transform.
Note that $\F_{0}$ is simply the identity while $\F_{\mathbf{1}} = \F_{(1, \ldots, 1)}$ equals the $d$-dimensional Fourier transform.
Every $\F_{\varpi}$ acts as an isomorphism on $\H_{0}(\R^{d})$, so by transposition also on $\H^{\prime}_{0}(\R^{d})$.
Moreover, by \cite[Proposition 7.1]{T-ImagFuncDistSpBargmannTrans}, any $f \in \H^{\prime}_{0}(\R^{d})$ belongs to $\S_{1/2}(\R^{d})$ if and only if
	\[ |\F_{\varpi}\{f\}(\xi)| \lesssim e^{- \eta |\xi|^{2}} , \qquad \varpi \in \{ 0, 1 \}^{d} , \]
for some $0 \leq \eta \leq 1/2$. Here $\S_{1/2}(\R^d)$ denotes the the smallest $d$-dimensional Fourier invariant classical Gelfand-Shilov space \cite{Gelfand-ShilovII}.

We are now in a position to state the multidimensional analogs of our main Theorems \ref{t:TFCharH01/2} and \ref{t:FourierCharProperPilipovic}.
Before doing so, let us briefly comment on their proofs.
Consider for any $\varpi \in \{0, 1\}^{d}$ the transform 
	\[ U_{\varpi}(x, \xi) =  ((\mathbf{1} - \varpi) \odot x + \varpi \odot \xi, - \varpi \odot x + (\mathbf{1} - \varpi) \odot \xi) , \qquad x, \xi \in \R^{d} , \]
where $\odot$ denotes the Hadamard product of two vectors.
Then one verifies that for the \emph{$d$-dimensional short-time Fourier transform} \cite{GrochenigBook}
 	\[ V_{\phi} f(x, \xi) = (2 \pi)^{-d/2} \int_{\R^{d}} f(t) \phi(t - x) e^{-i \ev{t}{\xi}} dt \] 
with $\phi(x) = h_{0}(x) = \pi^{-d/4} e^{-\frac{1}{2} |x|^{2}}$, it holds that
	\[ V_{\phi} [\F_{\varpi} f](x, \xi) = e^{- i \ev{\varpi}{x \odot \xi}} V_{\phi} f (U_{\varpi}(x, \xi)) . \]
Using the multidimensional analog of \eqref{eq:STFTBargmann} (see \cite[(B.2)]{T-G-FourierCharPilipovicSp}), connecting the $d$-dimensional Bargmann transform and the short-time Fourier transform, one finds similar results to
 Lemmas \ref{l:TFDecayToBargmann} and \ref{l:BargmannToTFDecay}, linking the decay of the $\F_{\varpi}\{f\}$ and the bounds on the Bargmann transform (see also \cite[Section 2]{T-G-FourierCharPilipovicSp}).
By proving a multidimensional variant of Theorem \ref{t:PhragmenLindelofChar} (using recursively the one-dimensional case, see \cite[Proposition A.2]{T-G-FourierCharPilipovicSp}), one may then obtain the ensuing theorems as before.
We remark that this analysis is enabled by the condition $(\alpha)$ on $\omega$ and the convexity of $\varphi^{*}$.

	\begin{theorem}
		A $f \in \H^{\prime}_{0}(\R^{d})$ satisfies
			\[ |\F_{\varpi}\{f\}(\xi)| \lesssim e^{-(\frac{1}{2} - \lambda) |\xi|^{2}} , \qquad \varpi \in \{0, 1\}^{d}, ~ \xi \in \R^{d} , \]
		for every $\lambda > 0$ if and only if $f \in \H_{0, 1/2}(\R^{d})$.
	\end{theorem}
	
	\begin{theorem}
		Let $\omega$ be a weight function satisfying $(\beta_{2}^{*})$ and $(\delta)$. Then, $f \in \H^{\prime}_{0}(\R^{d})$ satisfies
			\[ |\F_{\varpi}\{f\}(\xi)| \lesssim e^{-\frac{1}{2}  |\xi|^{2} + \lambda \omega(|\xi|)}, \qquad \varpi \in \{0, 1\}^{d}, ~ \xi \in \R^{d} , \]
		for some $\lambda > 0$ (for every $\lambda > 0$) if and only if $f \in \H_{\omega}(\R^{d})$ ($f \in \H_{0, \omega}(\R^{d})$).
	\end{theorem}
	
Once again, we find that $\H_{s}(\R^{d}) = \H_{\omega_{s}}(\R^{d})$ and $\H_{0, s}(\R^{d}) = \H_{0, \omega_{s}}(\R^{d})$ for $s \in \overline{\R}_{\flat}$, $s \leq 1/2$, and $\omega_{s}$ as in \eqref{eq:WeightFuncClassicalPilipovicSp},
where $\H_{s}(\R^{d})$ and $\H_{0, s}(\R^{d})$ are the $d$-dimensional proper Pilipovi\'{c} spaces as introduced in \cite{T-ImagFuncDistSpBargmannTrans}, see Example \ref{ex:ClassicalPilipovicSp}.
Here, we have defined $\H_{0, 0}(\R^{d}) = \spn \{ \phi \}$.
Then the ensuing corollary gives again a significant improvement to
 \cite[Theorem 2.2 and Theorem 2.3]{T-G-FourierCharPilipovicSp}.
	
	\begin{corollary}
		\label{t:TFCharClassicalPilipovicSp}
		Let $s \in \overline{\R}_{\flat}$ be such that $0 \leq s < 1/2$ (such that $0 \leq s \leq 1/2$).
		For any $f \in \H^{\prime}_{0}(\R^{d})$, then
		$f \in \H_{s}(\R^{d})$ ($f \in \H_{0, s}(\R^{d})$) if and only if
			\[ |\F_{\varpi}\{f\}(\xi)| \lesssim e^{-\frac{1}{2} |\xi|^{2} + \lambda \omega_{s}(|\xi|)} ,  \qquad \varpi \in \{0, 1\}^{d}, ~ \xi \in \R^{d} , \]
		for some $\lambda > 0$ (for every $\lambda > 0$), where  the weight function $\omega_{s}$ is given by \eqref{eq:WeightFuncClassicalPilipovicSp}.
	\end{corollary}
	
We end this article by mentioning that in the two specific cases of $\H_{0}(\R^{d})$ and $\H_{0, 0}(\R^{d})$, the previous result may be further improved by only requiring appropriate decay of the function and its $d$-dimensional Fourier transform.
Indeed, for the case of $\H_{0, 0}(\R^{d})$ this is trivial, while that of $\H_{0}(\R^{d})$ follows from \cite[Theorem 1.3]{BDJ2003}.
It would be interesting to determine whether or not similar results are also true for the other spaces in the above theorems.

\end{document}